\documentclass[reqno]{amsart}
\usepackage{amsmath}
\usepackage{amsthm,amssymb}
\usepackage{color,comment,soul}
\usepackage{tikz}

\let\epsilon=\varepsilon
\let\phi=\varphi
\let\emptyset=\varnothing

\newcommand{\R}{\mathbb{R}}

\newcommand{\N}{\mathbb{N}}

\newcommand{\sign}{\operatorname{sgn}}

\newcommand{\extr}{\operatorname{ext}}
\newcommand{\inter}{\operatorname{int}}
\newcommand{\Int}{\operatorname{int}}

\newcommand{\cl}{\operatorname{cl}}
\newcommand{\Log}{\operatorname{Log}}
\newcommand{\Exp}{\operatorname{Exp}}
\newcommand{\id}{\operatorname{id}}

\newcommand{\conv}{\operatorname{conv}}

\newcommand{\Fix}{\operatorname{Fix}}

\newcommand{\var}[1]{\| #1 \|_{\mathrm{v}}}

\numberwithin{equation}{section}

\theoremstyle{plain}
\newtheorem{theorem}{Theorem}[section]

\newtheorem{corollary}[theorem]{Corollary}
\newtheorem{lemma}[theorem]{Lemma}
\newtheorem{proposition}[theorem]{Proposition}

\theoremstyle{definition}

\newtheorem{example}[theorem]{Example}

\theoremstyle{remark}
\newtheorem{remark}[theorem]{Remark}

\begin{document}
\title[Detecting fixed points by illuminating the unit ball]{Detecting fixed points of nonexpansive maps by illuminating the unit ball}
\author[B. Lemmens, B. Lins, R. Nussbaum]{Bas Lemmens, Brian Lins$^*$, Roger Nussbaum$^\dagger$}
\date{}
\address{Brian Lins, Hampden-Sydney College}
\email{blins@hsc.edu}
\address{Bas Lemmens, University of Kent}
\email{B.Lemmens@kent.ac.uk}
\address{Roger Nussbaum, Rutgers University}
\email{nussbaum@math.rutgers.edu}
\thanks{$^*$Corresponding author.}
\thanks{$^\dagger$This work was partially supported by NSF grant DMS-1201328}

\begin{abstract} 
We give necessary and sufficient conditions for a nonexpansive map on a finite dimensional normed space to have a nonempty, bounded set of fixed points. Among other results we show that if $f : V \rightarrow V$ is a nonexpansive map on a finite dimensional normed space $V$, then the fixed point set of $f$ is nonempty and bounded if and only if there exist $w_1, \ldots , w_m$ in $V$ such that $\{f(w_i) - w_i : i = 1, \ldots, m \}$ illuminates the unit ball. This yields
a numerical procedure for detecting fixed points of nonexpansive maps on finite dimensional spaces. We also discuss applications of this procedure to certain nonlinear eigenvalue
problems arising in game theory and mathematical biology. 
\end{abstract}

\maketitle

\section{Introduction}\label{sec:intro}

A central problem in metric fixed point theory is to understand when a nonexpansive map $f:X \rightarrow X$ on a metric space $(X,d)$ has a fixed point.  There are numerous results when $X$ is a closed, bounded, convex subset in a Banach space $V$, see \cite{GoebelKirk,GoebelReich}.  Of course, if $V$ is finite dimensional and $X \subset V$ is compact and convex, then the Brouwer fixed point theorem immediately resolves the question.  If, however, $X$ is unbounded, it is not at all clear when $f$ has a fixed point, even if the normed space $V$ is finite dimensional.  

In this paper we study the fixed point set, $\Fix(f)$, of nonexpansive maps on finite dimensional normed spaces. For such maps, there are many algorithms known to approximate fixed points \cite{Ishikawa, Mann}, if one exists.  The results of this paper complement these algorithms by providing computational methods that can confirm the existence of fixed points.  The paper is organized as follows. 

In Section \ref{sec:horo} below we review the horofunction compactification of a complete, proper, metric space. Our Proposition \ref{prop:beardon} extends a result of Beardon \cite[Proposition 4.5]{Beardon97} and allows us to give necessary and sufficient conditions for the fixed point set of a nonexpansive map to be bounded and nonempty.

In Section \ref{sec:illum} we focus on the fixed point set of maps $f:V \rightarrow V$ when $V$ is a finite dimensional normed space.  Our main result, Theorem \ref{thm:main}, gives necessary and sufficient conditions for $\Fix(f)$ to be nonempty and bounded.  In particular, we show that $\Fix(f)$ is nonempty and bounded if and only if there exist a finite number of points $w_1 , \ldots , w_m$ in $V$ such that $S := \{f (w_i ) - w_i : i = 1, \ldots , m\}$
illuminates the closed unit ball $B_1$ of $V$. Recall that $S$ illuminates $B_1$ if for each $w \in \partial B_1$ there exists $s \in S$ such that $w + \lambda s \in \inter B_1$ for all $\lambda > 0$ sufficiently small. Interestingly, it is a famous unresolved conjecture whether every compact, convex body in an $n$-dimensional vector space $V$ can be illuminated by some subset $S$ of $V$ with cardinality less than or equal to $2^n$. 

Theorem \ref{thm:main} suggests the following simple procedure for detecting fixed points of nonexpansive
maps $f : V \rightarrow V$. Generate randomly a finite set $S$ in $V$ and check if $\{f (w) - w : w \in S\}$ illuminates the unit ball of $V$. In Section \ref{sec:norm} we discuss criteria that can be verified computationally to check if a set illuminates the unit ball for a variety of norms. In Section \ref{sec:PF}, we apply the results to certain nonlinear eigenvalue problems that arise in game theory \cite{BewleyKohlberg,RosenbergSorin} and mathematical biology \cite{Nussbaum89,Schoen86}, and perform some numerical experiments to test the feasibility of the procedure. Of particular interest is Theorem 5.1 which is a nonlinear Perron-Frobenius theorem.  The final section explains how illuminating sets can be used to place bounds on the location of the fixed point set of nonexpansive maps.

\section{Fixed points and horofunctions} \label{sec:horo}
Throughout the paper we will use the notation $\cl A$, $\inter A$ and $\partial A$ to, respectively, denote the closure, interior and boundary of a set $A$ in a metric space. We will also denote the closed ball with radius $r>0$ and center $x$ by $B_r(x)$. 

Let us briefly recall the horofunction compactification of a complete, proper, metric space $(X,d)$, see \cite{BridsonHaefliger,Gromov,Rieffel}. Here, \textit{proper} means that every closed ball of radius $R \ge 0$ in $X$ is compact. Let $C(X)$ denote the space of continuous functions $f\colon X\to \mathbb{R}$ equipped with the topology of compact convergence,  see \cite[\S 46]{Munkres}.  
Fix a {\em base point} $b\in X$ and define for $x\in X$ the function  $\tau_b(x)\colon X\to\mathbb{R}$ by 
\[
\tau_b(x)(y) := d(y,x)-d(b,x) \quad\text{for all } y\in X.
\]
It is easy to check that for each $x\in X$ the function $\tau_b(x)$ is Lipschitz with constant 1, and hence $\tau_b(X)\subseteq C(X)$ is a an equicontinuous family. Moreover, for each fixed $y\in X$ we have that $\{\tau_b(x)(y)\colon x\in X\}\subseteq [-d(y,b),d(y,b)]$. Thus, it follows from Ascoli's theorem \cite[Theorem 47.1]{Munkres} that $\tau_b(X)$ has compact closure in $C(X)$. 
The {\em horofunction boundary} of $(X,d)$ is given by $X(\infty):=\cl \tau_b(X) \setminus\tau_b(X)$, and its elements are called {\em horofunctions}. Given $h\in X(\infty)$ the set $H(h,r):=\{x\in X\colon h(x)\leq r\}$ is called the {\em horoball}  with center $h$ and radius $r\in\R$. As $(X,d)$ is proper, it is  $\sigma$-compact, i.e., $X$ is the union of countably many open  sets with compact closure, and hence the topology of compact convergence is metrizable, see \cite[Exercise 10, p. 289]{Munkres}. This implies that every sequence $\tau_b(x_n)$ in $C(X)$ has a convergent subsequence. 

If we furthermore assume that the  complete, proper, metric space $(X,d)$  is {\em geodesic}, that is to say, for each $x\neq y$ in $X$ there exists a path $\gamma\colon [\alpha,\beta]\to (X,d)$ such that $\gamma(\alpha)=x$, $\gamma(\beta) =y $ and $d(\gamma(s),\gamma(t))=|s-t|$ for all $\alpha\leq s\leq t\leq \beta$, then the horofunctions are precisely the limits of converging sequences $\tau_b(x_n)$ such that $d(b,x_n)\to\infty$. Indeed the following lemma holds, which is a slightly weaker result than \cite[Theorem 4.7]{Rieffel} by Rieffel, who showed that the horofunctions are precisely the limits of so called weakly geodesic rays. For completeness we include a proof.
\begin{lemma}\label{lem:2.1} 
If $(X,d)$ is a complete, proper, geodesic metric space, then $h\in X(\infty)$ if and only if there exists a sequence $(x_n)$ in $X$ with $d(b,x_n)\to \infty$ such that $\tau_b(x_n)$ converges to $h\in \cl \tau_b(X)$ as $n\to \infty$.
\end{lemma}
\begin{proof}
 If $h\in X(\infty)$, then there exists a sequence $(x_n)$ in $X$ such that  $\tau_b(x_n)\to h$, since the topology of compact convergence is metrizable whenever $(X,d)$ is proper. Note that $d(b,x_n)\to\infty$. Indeed, otherwise $(x_n)$ has a bounded subsequence $(x_{n_k})$ which converges to some  point say $x\in X$, as $(X,d)$ is a proper metric space. This implies that $h(y) =\lim_{k\to\infty} \tau_b(x_{n_k})(y)= \tau_b(x)(y)$ 
 for all $y\in X$,  and hence $h$ is not a horofunction, which is a contradiction.
 
To prove sufficiency, we argue by contradiction. Suppose that $h\not\in X(\infty)$. Then there exists $x_0\in X$ such that $h(y)=\tau_b(x_0)(y)$ for all $y\in X$. Let $r:=d(b,x_0)+1$ and $0<\epsilon<1$. As $\tau_b(x_n)\to h$, there exists $N\geq 1$ such that 
 \[
 \sup_{y\in B_r(x_0)}|\tau_b(x_n)(y)-h(y)|<\epsilon \mbox{\quad and\quad  }d(x_0,x_n)\geq r \mbox{\quad for all }n\geq N,
 \] 
 as the closed ball $B_r(x_0)$ is compact and $d(b,x_n)\to\infty$ as $n\to\infty$. Fix $n\geq N$ and let $\gamma_n\colon [0,\beta_n]\to X$ be a geodesic from $x_0$ to $x_n$. 
 Put $z:=\gamma_n(r)$, so $d(x_0,z)=r$, and note that
\begin{align*} 
 |\tau_b(x_0)(z)-h(z)| &\geq  |\tau_b(x_0)(z)-\tau_b(x_n)(z)|-|\tau_b(x_n)(z)-h(z)|\\
  &\geq  |d(z,x_0) -d(b,x_0) -d(z,x_n)+d(b,x_n)| -\epsilon\\
  &= |1-d(z,x_n)+d(b,x_n)|-\epsilon\\
  &\geq 1-d(z,x_n)+d(b,x_n)-\epsilon. 
\end{align*}
However,
$$d(b,x_n) \ge -d(b,x_0) +d(x_0,x_n) =-d(b,x_0) +d(x_0,z) +d(z,x_n) = 1+ d(z,x_n).$$
Substituting this lower estimate for $d(b,x_n)$ in the expression above gives a lower estimate for
$|\tau_b(x_0)(z)-h(z)|$ of $2-\epsilon > 1$, which is a contradiction. 
\end{proof} 

Recall that $f\colon X\to X$ is  a {\em nonexpansive}  map on a metric space $(X,d)$ if $d(f(x),f(y))\leq d(x,y)$ for all $x,y\in X$. 
For complete, proper, geodesic  metric space $(X,d)$ we introduce the following two properties: 
\begin{enumerate}
\item[A1.] For each nonexpansive map $f\colon X\to X$ there exists a sequence of  nonexpansive maps $f_n \colon X\to X$ such that each $f_n$ has a fixed point in $X$ and $f_n(x)\to f(x)$ as $n\to\infty$ for all $x\in X$. 
\item[A2.] If $f\colon X\to X$ is a nonexpansive map and there exists a closed ball $B$ in $(X,d)$ and a horoball $H(h,r)$ in $X$ such that $f(B\cap H(h,r))\subseteq B\cap H(h,r)$, then $f$ has a fixed point in $B\cap H(h,r)$. 
\end{enumerate}
All finite dimensional normed spaces satisfy (A1) and (A2). Indeed, their horoballs  are closed convex sets \cite{Walsh07}, so that (A2) follows from the Brouwer fixed point theorem. Furthermore, the maps $f_n:=(1-\frac{1}{n})f$ are Lipschitz contractions, which have unique fixed points, and hence (A1) holds. Other interesting metric spaces that satisfy the properties include,  Hilbert's metric spaces \cite{LNBook,Nussbaum88,HandbookHilbertGeometry,Walsh08}, Thompson's metric on finite dimensional cones \cite{LNBook,Nussbaum88,Thompson}, and hyperbolic spaces \cite{BridsonHaefliger}. Also note that a Busemann metric space satisfies property (A1). Indeed, if 
$(X,d)$ is Busemann, then we can fix $x_0\in X$ and define for each $x\in X$ an affinely re-parametrized geodesic $\gamma_x\colon [0,1]\to X$ connecting $x_0$ to $x$, so $d(\gamma_x(s),\gamma_x(t))=|t-s|d(x,x_0)$ for all $s,t\in [0,1]$. The geodesics $\gamma_x$ satisfy 
\begin{equation}\label{eq:bus}
d(\gamma_x(s),\gamma_y(s))\leq sd(x,y)\mbox{\quad  for all $x,y\in X$ and $s\in [0,1]$,}
\end{equation}
as $(X,d)$ is Busemann, see \cite[Proposition 8.1.2]{Papadopoulos}. It follows that the maps $r_s\colon X\to X$ given by, $r_s(x)=\gamma_x(s)$, are Lipschitz contractions. Thus, for each $n\geq 1$ the map $f_n\colon X\to X$ given by, $f_n(x) =f(r_{1-1/n}(x))$, is a Lipschitz contraction on $X$ that satisfies $d(f_n(x),f(x))\leq d(r_{1-1/n}(x),x) =d(\gamma_x(1-1/n),\gamma_x(1)) = \frac{1}{n}d(x,x_0)\to 0$, as $n\to\infty$, for all $x\in X$, and hence property (A1) holds. In fact, one does not need all geodesics to satisfy the Busemann convexity property (\ref{eq:bus}) as long as there are enough to define the Lipschitz contractions $r_s$. Such metric spaces have been studied in \cite{GaubertVigeral} and are called metrically star-shaped. 
  
The following proposition extends a result by Beardon \cite[Proposition 4.5]{Beardon97}. The reader should note that Beardon makes some additional assumptions that are not required in our setting. 
\begin{proposition}\label{prop:beardon} 
If $f\colon X\to X$ is a nonexpansive map on a complete, proper, geodesic metric space $(X,d)$ satisfying properties (A1) and (A2), then $\Fix(f)$ is empty  or unbounded if and only if there exists $h\in X(\infty)$ such that $h(f(x))\leq h(x)$ for all $x\in X$. 
\end{proposition}
\begin{proof}
 The assertion that if  $\Fix(f)=\emptyset$ or unbounded, then there exists $h\in X(\infty)$ such that $h(f(x))\leq h(x)$ for all $x\in X$ is due to Beardon \cite[Proposition 4.5]{Beardon97}. For completeness we include the argument. 
  Suppose that $\Fix(f)=\emptyset$. As $(X,d)$ satisfies (A1), there exists a sequence of  nonexpansive maps $f_n \colon X\to X$ such that each $f_n$ has a fixed point $x_n\in X$ and $f_n(x)\to f(x)$ as $n\to\infty$ for all $x\in X$. Note that $d(b,x_n)\to\infty$ as $n\to\infty$, as otherwise  there exists a convergent subsequence $(x_{n_k})$, with limit say $z$, as $(X,d)$ is proper. Clearly 
  \[
  d(z,f(z))\leq 
  d(z,x_{n_k})+d(f_{n_k}(x_{n_k}),f_{n_k}(z))+d(f_{n_k}(z),f(z))\to 0\mbox{\quad as }k\to\infty,
 \] 
 so that $z$ is a fixed  point of $f$, which is impossible. 
 
 By taking a subsequence we may assume that $\tau_b(x_n)$ converges to $h\in X(\infty)$ by Lemma \ref{lem:2.1}. 
 Note that for each $x\in X$ we have that 
\begin{align*}
d(f(x),x_n)-d(b,x_n) &\leq d(f(x),f_n(x)) +d(f_n(x),f_n(x_n))-d(b,x_n)\\
&\leq d(f(x),f_n(x))+d(x,x_n)-d(b,x_n).
 \end{align*}
 By taking limits, we deduce that $h(f(x))\leq h(x)$ for all $x \in X$. If $\Fix(f)$ is unbounded, then we can use an unbounded sequence of fixed points $(x_n)$ of $f$ to create a horofunction $h$ such that $h(f(x))\leq h(x)$ for all $x\in X$. 
 
To prove the converse statement suppose that there exists a horofunction $h$ with $h(f(x))\leq h(x)$ for all $x\in X$. By Lemma \ref{lem:2.1} there exists a sequence $(z_n)$ in $X$ such that $\tau_b(z_n)\to h$ and $d(z_n,b)\to\infty$. Consider the horoballs $H_{-r}:=H(h,-r)=\{x\in X\colon h(x)\leq -r\}$ for $r\geq 0$. Note that $H_{-r}$ is nonempty. Indeed, for each $n\geq 1$  such that $d(b,z_n)>r$,  there exists 
$y_n\in \partial B_r(b)$  such that  $d(z_n,b) = d(z_n,y_n)+d(y_n,b)$, as $(X,d)$ is a geodesic space. As $\partial B_r(b)$ is compact, there exists a subsequence $(y_{n_k})$ such that $y_{n_k}\to y^*$. Thus, for each $\epsilon >0$ there exists $K\geq 1$ such that $h(y^*)-\epsilon < \tau_b(z_{n_k})(y^*)$ and 
$d(y^*,y_{n_k})<\epsilon$ for all $k\geq K$. We have  that 
\begin{align*}
\tau_b(z_{n_k})(y^*) &= d(y^*,z_{n_k})-d(z_{n_k},b) \\
&\leq d(y^*,y_{n_k}) + d(y_{n_k},z_{n_k}) -d(z_{n_k},b) \\
&<\epsilon -d(y_{n_k},b) =\epsilon -r,
\end{align*}
which implies that $h(y^*)\leq -r$.

Note that if $y\in H_{-r}$, then for each $\epsilon >0$ small and each $n$ sufficiently large we have that 
 \[
 d(y,z_n)-d(b,z_n)\leq -r+\epsilon,
 \]
 so that $d(y,b)\geq d(b,z_n)-d(y,z_n)\geq r-\epsilon$. Thus, $d(b,y)\geq r$ for all $y\in H_{-r}$. 
 
 Now suppose that $\Fix(f)$ is nonempty, and let $z\in \Fix(f)$. For $m\geq 1$ pick $w_m\in H_{-m}$ and put $r_m:=d(w_m,z)$. Then we know that $f(B_{r_m}(z)\cap H_{-m})\subseteq B_{r_m}(z)\cap H_{-m}$, for all $m\geq 1$. By property (A2), $f$ has a fixed point $v_m\in B_{r_m}(z)\cap H_{-m}$ for each $m\geq 1$. But then $d(v_m,b)\geq m$, which shows that $\Fix(f)$ is unbounded, which completes the proof.
\end{proof}

For a finite dimensional normed space $(V,\|\cdot\|)$, let $(V^*,\|\cdot\|_*)$ denote the dual space. Recall that a norm is \textit{smooth} if for every $z \in V$ with $\|z\|=1$, there is a unique $\phi \in V^*$, $\|\phi\|_* = 1$, such that $\phi(z) = 1$. It is a consequence of \cite[Proposition 6.2]{Rieffel} that for a finite dimensional vector space with a smooth norm, the horofunctions are given by $h(x) =-\phi(x-b)$ for $x\in V$, where $b \in V$ and $\phi\in V^*$ are fixed with $\|\phi\|_*=1$. This gives rise to the following generalization of \cite[Theorem 7]{Beardon90}. 
\begin{corollary}\label{cor:2.3}
If $f\colon V\to V$ is a nonexpansive map on a finite dimensional normed space $(V,\|\cdot\|)$ with a smooth norm, then $\Fix(f)$ is empty or unbounded if and only if there exists a nonzero linear functional $\phi\in V^*$ such that $\phi(f(x))\geq \phi(x)$ for all $x\in V$.
\end{corollary}
\begin{remark}
Part of Corollary \ref{cor:2.3} does not depend on the smoothness of the norm.  
Indeed, suppose that $C$ is a closed, convex subset of a finite dimensional normed linear space $(V,\|\cdot\|)$, where the norm is not necessarily smooth.  Let $f\colon C \to C$ be a nonexpansive map and suppose that $\phi$ is a nonzero linear functional on $V$ such that (i) $\{x \in C \colon \phi(x) \ge a \}$ is nonempty for all $a$, and (ii) $\phi(f(x)) \ge \phi(x)$ for all $x \in V$.  Then $\Fix(f)$ is empty or unbounded.  

To show this, let $H_m := \{x \in C \colon \phi(x) \ge m \}$, $m \in \N$, so $H_m$ is closed, convex and nonempty. Assume $\Fix(f)$ is nonempty and take $z \in \Fix(f)$ and $w_m \in H_m$ and define $r_m := \|z-w_m\|$.  By our construction $w_m \in B_{r_m}(z) \cap H_m$ and $B_{r_m}(z) \cap H_m$ is closed, bounded and convex. Because $f$ is nonexpansive and condition (ii) above is satisfied, $f(B_{r_m}(z) \cap H_m) \subseteq B_{r_m} \cap H_m$, so Brouwer's fixed point theorem implies that $f$ has a fixed point $x_m \in B_{r_m} \cap H_m$.  Since $\phi(x_m) \ge m$, $\|x_m\| \ge m/\|\phi\|$ and $\Fix(f)$ is unbounded.  

\end{remark}

Proposition \ref{prop:beardon} gives a criterion for the existence of a fixed point in terms of horofunctions and one may wonder for which metric spaces this criterion can be verified computationally. In this paper we will see that for finite dimensional normed spaces there is a nice way to check the criterion by using so-called illuminating vectors.

\section{Fixed points in normed spaces} \label{sec:illum}
Given a compact convex set $K$ in a finite dimensional vector space $V$ with nonempty interior, we say that $x\in\partial K$ is {\em illuminated by} $w\in V$ if $x+\lambda w\in \inter K$ for some $\lambda>0$. A set $S\subseteq V$ is said to {\em illuminate } $K$ if each point in $\partial K$ is illuminated by some $w\in S$.  The {\em illumination number} of $K$ is defined by 
\[
c(K):=\min\{|S|\colon S\subset V \text{ illuminates } K\}.
\]
The illumination number was introduced by Boltjanski \cite{Boltjanksi}, who showed that it is equal to the so-called {\em covering number} $b(K)$ of $K$, which is the smallest number of strictly smaller homothetical copies $K_1,\ldots,K_m$ of $K$ which cover $K$, so $K\subseteq \cup_{i=1}^m K_i$. It is clear by compactness that $b(K)$ is a finite number. Gohberg and Markus \cite{GohbergMarkus} conjectured that  $b(K)\leq 2^n$ for any compact, convex body $K$ in $V$, where $n=\dim (V)$. Moreover, equality holds if and only if $K$ is an $n$-dimensional parallelepiped. Hadwiger \cite{Hadwiger57} also independently raised the question of the maximal value of $b(K)$. Gohberg and Markus's conjecture is commonly referred to as the Illumination Conjecture, and remains unsolved for general compact convex sets.  A detailed survey is given in \cite[Chapter VI]{BoltyanskiMartiniHorst}.

Before stating our main result, let us recall the definition of the topological degree  in finite dimensional vector spaces $V$. Given an open, bounded set $G\subseteq V$ and a continuous map $f\colon \cl G\to V$ such that $f(x)\neq a$ for all $x\in\partial G$, there exists an integer $\deg(f,G,a)$ called the {\em topological degree of $f$ on $G$ with respect to $a$}, which has the following properties: 
\begin{enumerate}
\item[D1.] $\deg(\id, G,a) = 1$ if $a\in G$ and $\deg(\id,G,a)=0$ if $a\not\in \cl G$, where $\id$ is the identity map on $V$. 
\item[D2.] (Additivity Property) If $G_1$ and $G_2$ are disjoint, open subsets of $G$ and if $a\not\in f(\cl G\setminus(G_1\cup G_2))$, then 
\[
\deg(f,G,a) = \deg(f,G_1,a)+\deg(f,G_2,a). 
\]
\item[D3.] (Homotopy Property) Let $F\colon \cl G\times [0,1]\to V$ be a continuous map, and let $f_t\colon \cl G\to V$ be given by $f_t(x)=F(x,t)$ for all $t\in [0,1]$. If $a\not\in F(\partial G\times [0,1])$, then $\deg(f_t,G,a)$ is constant for all $t\in [0,1]$. 
\end{enumerate} 

We allow the possibility that $G$, $G_1$, or $G_2$ is empty, so (D2) gives $\deg(f,\emptyset, a)=0$. Notice this implies that if $\deg(f,G,a)\neq 0$, there exists $x\in \cl G$ with $f(x) =a$, but \emph{not} conversely. Amann and Weiss \cite{AmannWeiss} proved that properties (D1), (D2) and (D3) uniquely determine the topological degree. We refer the reader to \cite{Deimling} for further details.

The following proposition is almost certainly known, but we are unaware of a reference. 
\begin{proposition}\label{prop:2.4} Let $G$ be an open, bounded subset of a finite dimensional normed space $(V,\|\cdot\|)$, and suppose that $g\colon \cl G \to V$ is a continuous map with $g(x)\neq 0$ for all $x\in\partial G$. If $\deg(g,G,0)\neq 0$, then for each $b\in V$ with $\|b\|=1$ there exists  $x\in\partial G$ such that $b=g(x)/\|g(x)\|$. 
\end{proposition}
\begin{proof}
We argue by contradiction. So, suppose that there exists $b\in V$ with $\|b\|=1$ such that $g(x)\neq \|g(x)\|b$ for all $x\in\partial G$. Then for each $x\in\partial G$ and each $\lambda>0$ we have that $g(x)\neq \lambda b$. Now define for $t\in [0,1]$ a continuous map $g_t\colon \cl G\to V$ by 
\[
g_t(x)=(1-t)g(x)-tb \text{ for all }x\in\cl G.
\] 
Note that for $0\leq t<1$ we have that $g_t(x)\neq 0$ for all $x\in\partial G$, as otherwise $g(x) =\frac{t}{1-t}b$ for some $x\in\partial G$. Also $g_1(x)=-b\neq 0$ for all $x\in \cl G$. So, the homotopy property (D3) gives $\deg(g,G,0)=\deg(g_1,G,0)=0$, which is a contradiction.
\end{proof}

We also have the following proposition. 
\begin{proposition}\label{prop:2.5}  Let $D$ be an open subset of a finite dimensional normed space $(V,\|\cdot\|)$. Suppose that  $f\colon D\to V$ is a nonexpansive map such that $\Fix(f)$ is compact and nonempty. Then $\Fix(f)$ is connected. If $G$ is a bounded open set such that $\Fix(f)\subset G$ and $\cl G \subset D$, then $\deg(\id-f,G,0)=1$. 
\end{proposition}

\begin{proof}
By our assumptions $x \ne f(x)$ for all $x \in \partial G$ since $\Fix(f) \subset G$ and $G$ is open. Therefore $\deg(\id-f,G,0)$ is defined.   More generally, suppose that $H$ is a bounded open subset such that $H \cap \Fix(f)$ is nonempty, $\cl H \subset D$, and $f(x) \neq x$ for all $x \in \partial H$.  Because $\partial H$ is a compact set and $f$ is a continuous map such that $\|x-f(x)\| >0$ for all $x\in \partial H$, there exists $c > 0$ with $\|x-f(x)\| \ge c$ for all $x \in \partial H$.

By assumption, there exists $x_0 \in \Fix(f)\cap H$. For $0 \le t \le 1$ define $f_t (x):=(1-t)f(x)+tx_0$. Note that 
$$\|x-f(x)-(x-f_t(x))\|= t\|f(x)-x_0\|=t\|f(x)-f(x_0)\| \le t\|x-x_0\|.$$
It follows that there exists $\delta > 0$ such that for all $x \in \partial H$ and all $t$ with $0 \le t \le \delta$, 
$$\|x-f(x)-(x-f_t(x))\|\le \delta \|x-x_0\| < c.$$
The homotopy property (D3) now implies that 
$$\deg(\id-f,H,0) = \deg(\id-f_\delta,H,0).$$
Because $f$ is nonexpansive,
$$\|f_\delta(x) - f_\delta(y)\| \le (1-\delta) \|x-y\|$$
for all $x, y \in \cl H$, which implies that $f_\delta$ has at most one fixed point in $H$.  However, by our construction, we have $f_\delta(x_0) = x_0$.  If we select a number $r>0$ such that $B_r(x_0) \subseteq H$, the additivity property (D2) implies that 
$$\deg(\id-f_\delta,H,0) = \deg(\id-f_\delta,\inter B_r(x_0),0 ).$$

For $x \in \partial B_r(x_0)$ and $\delta \le t \le 1$, we have that 
\begin{align*}
\|x-f_t(x)\| &= \|x-x_0-(f_t(x)-f_t(x_0))\| \\
&\ge r - \|f_t(x)-f_t(x_0)\| \\
&\ge r-(1-t)\|x-x_0\| = tr > 0,
\end{align*}
so the homotopy property (D3) implies that 
$$\deg(\id-f_\delta,\inter B_r(x_0),0) = \deg(\id-x_0,\inter B_r(x_0),0) = 1,$$
where we have also used (D1).  Thus, under our assumptions on $H$, we have proved that $\deg(\id-f,H,0) = 1$.

It is clear that $G$ can be selected as in the statement of Proposition \ref{prop:2.5}, and as previously noted, $x - f(x) \ne 0$ for $x \in \partial G$.  Thus it follows from the above results that $\deg(\id-f,G,0) = 1$. To complete the proof, we argue by contradiction and assume that $\Fix(f)$ is not connected.   It follows that $\Fix(f) = A \cup B$, where $A$ and $B$ are disjoint, nonempty, compact sets.  It follows that there exist disjoint, bounded, open sets $H_A$ and $H_B$ with $A \subset H_A$, $B\subset H_B$, $\cl H_A \subset D$ and $\cl H_B \subset D$.  Our previous arguments imply that $\deg(\id-f,H_A,0) = \deg(\id-f,H_B,0) = 1$. Writing $H = H_A \cup H_B$, we also have that $\deg(\id-f,H,0) = 1$. However, the additivity property (D2) implies that 
$$1 = \deg(\id-f,H_A \cup H_B,0) = \deg(\id-f,H_A,0) + \deg(\id-f,H_B,0) =2,$$
a contradiction.    
\end{proof}

If the set $D$ in Proposition \ref{prop:2.5} is convex and $f$ is nonexpansive on $\cl D$, Bruck \cite{Bruck} has proved that $\Fix(f)$ is a nonexpansive retract of $D$, which gives connectedness of $\Fix(f)$ as a very special case.  A simple proof of this result, valid for the finite dimensional case, is given in \cite[\S 4]{Nussbaum98}.

Combining Propositions \ref{prop:2.4} and \ref{prop:2.5} gives the following corollary. 
\begin{corollary}\label{cor:everyDirection} Let $D$ be an open subset of a finite dimensional normed space $(V,\|\cdot\|)$. Suppose that  $f\colon D\to V$ is nonexpansive and $\Fix(f)$ is compact and nonempty. If $G$ is a bounded open set such that $\cl G \subset D$ and $\Fix(f) \subset G$, then for each $y\in V\setminus\{0\}$ there exists $w\in\partial G$ and $\lambda > 0$ such that $w-f(w) = \lambda y$. 
\end{corollary}
Note that if $f$ in Corollary \ref{cor:everyDirection} is defined on the whole of $V$ and $p \in V$, we can take $G$ to be any open ball with center $p$ and sufficiently large radius.

We now state the main result of this section.
\begin{theorem}\label{thm:main}
If $f\colon V\to V$ is a nonexpansive map on a finite dimensional normed space $(V,\|\cdot\|)$, then the following statements are equivalent: 
\begin{enumerate}
\item $\Fix(f)$ is nonempty and bounded. 
\item There exists a bounded open subset $G \subset V$ such that $x - f(x) \neq 0$ for all $x \in \partial G$ and $\deg(\id-f,G,0) = 1$.  
\item For every $y \in V \backslash \{0\}$, there exists $w \in V$ and $\lambda > 0$ such that $w-f(w) = \lambda y$.  
\item There exist $w_1,\ldots,w_m$ in $V$ such that $\{f(w_i)-w_i\colon i=1,\ldots,m\}$ illuminates $B_1(0)$. 
\item There exists no horofunction $h$  of $(V,\|\cdot\|)$ such that $h(f(x))\leq h(x)$ for all $x\in V$. 
\end{enumerate} 
\end{theorem}
\begin{proof} The statement (1) implies (2) by Proposition \ref{prop:2.5}. The implication (2) implies (3) follows from Proposition \ref{prop:2.4}. Also note that (3) implies (4), as the illumination number of $B_1(0)$ is finite. Proposition \ref{prop:beardon} gives that (5) implies (1). All that remains is to prove that (4) implies (5). 

Let $h$ be a horofunction in the horofunction compactification of $(V,\|\cdot\|)$ with base point $b$. We begin by observing that there exists $z \in V$, $\|z\|=1$, such that 
\begin{equation} \label{eq:horopath}
h(x+tz) = h(x) - t
\end{equation} 
for all $x \in V$ and $t \in \R$.  Indeed, by Lemma \ref{lem:2.1} there is a sequence $(y_k)$ in $V$ such that 
$$h(x) = \lim_{k \rightarrow \infty} \|x-y_k\| - \|b- y_k\|$$
and $\|y_k\| \rightarrow \infty$.  By passing to a subsequence if necessary, let $z = \lim_{k \rightarrow \infty} y_k/\|y_k\|$. The equation 
\begin{align*}
\left\|x + t\frac{y_k-x}{\|y_k-x\|} - y_k \right\| - \|b-y_k\| &= \| x - y_k \| \left( 1 - \frac{t}{\|x-y_k\|} \right) - \|b-y_k\| \\ 
&= \|x - y_k \| - \|b-y_k \| - t
\end{align*}
becomes equation \eqref{eq:horopath} in the limit as $k \rightarrow \infty$.  

Now suppose that $\{f(w_i)-w_i\colon i=1,\ldots,m\}$ illuminates the unit ball $B_1(0)$ in $V$. Since $z \in \partial B_1(0)$, there is some $i \in \{1,\ldots,m\}$ such that $f(w_i)-w_i$ illuminates $z$.  Therefore $z+\epsilon(f(w_i)-w_i) \in \inter B_1(0)$ for some sufficiently small $\epsilon>0$.  By translation, $z+w_i + \epsilon(f(w_i)-w_i) \in \inter B_1(w_i)$.  

As the pointwise limit of a sequence of convex, Lipschitz 1 functions, $h$ must be convex and Lipschitz 1 itself. Since $z+w_i + \epsilon(f(w_i)-w_i) \in \inter B_1(w_i)$, it follows that 
$$h(z+ \epsilon f(w_i) + (1-\epsilon)w_i) > h(w_i)-1  = h(z+w_i)$$
by equation \eqref{eq:horopath}. By convexity, we obtain from the preceeding inequality that 
$$(1-\epsilon)h(z+w_i) + \epsilon h(z+f(w_i)) > h(z+w_i),$$
which immediately gives $h(z+f(w_i)) > h(z+w_i)$.  
Then by equation \eqref{eq:horopath}, $h(f(w_i)) > h(w_i)$.
\end{proof}

Theorem \ref{thm:main} has the following interesting consequence concerning the space $N(V,\|\cdot\|)$ consisting of all nonexpansive maps on  a finite dimensional normed space $(V,\|\cdot\|)$.
\begin{corollary}\label{cor:generic} 
The subset of $N(V,\|\cdot\|)$ consisting of those nonexpansive maps $f\colon V\to V$  with $\Fix(f)$ nonempty and bounded, is open and dense in the topology of compact convergence on $N(V,\|\cdot\|)$. Moreover, if $f\in N(V,\|\cdot\|)$ is such that $\Fix(f)$ is unbounded, then for each $\delta>0$ there exists $g\in N(V,\|\cdot\|)$ such that $\Fix(g)$ is empty and 
\[
\sup_{x\in V} \|f(x)-g(x)\|\leq \delta.
\]
\end{corollary}
\begin{proof}
If $\Fix(f) \subset \inter B_R(0)$ is nonempty, then $\deg(\id-f, \inter B_R(0), 0)=1$ by Proposition \ref{prop:2.5}. Select $\epsilon >0$ such that $\min\{\|x-f(x)\|\colon x\in\partial B_R(0)\}\ge\epsilon$ and consider the neighborhood 
\[
U:=\Big{\{}h\in N(V,\|\cdot\|)\colon \sup_{x\in\partial B_R(0)} \|f(x)-h(x)\|<\epsilon/2\Big{\}}
\]
of $f$ in the topology of compact convergence. Let $g\in U$ and define the homotopy $g_t(x):=tg(x) +(1-t)f(x)$ for $t\in [0,1]$ and $x\in B_R(0)$. Then for $x\in\partial B_R(0)$ we  have that 
\[
\|x-g_t(x)\|\geq \|x-f(x)\|-\|f(x)-g_t(x)\|\geq \epsilon -t\|f(x)-g(x)\|>\epsilon/2 >0
\] 
for all $t\in[0,1]$. So, by (D3) we get that $\deg(\id-g, \Int B_R(0),0)=1$, and hence (D1) implies  that $g$ has a fixed point in $\Int B_R(0)$. As $\Fix(g)$ is connected, see \cite[Theorems 2 and 3]{Bruck}, and $g$ has no fixed points in $\partial B_R(0)$, we conclude that $\Fix(g)\subseteq B_R(0)$, which shows that
$$\{ f\in N(V,\|\cdot\|)\colon \Fix(f)\mbox{ nonempty and bounded}\}$$ 
is open. To show that it is dense let $f\in N(V,\|\cdot\|)$ and recall that the topology of compact convergence has a basis of open sets,  
\[
U(g,A,\epsilon):=\{h\in N(V,\|\cdot\|)\colon \sup_{x\in A}\|h(x)-g(x)\|<\epsilon\},
\]
where $g\in N(V,\|\cdot\|)$, $A\subseteq V$ compact, and $\epsilon >0$. Let $f_k(x):=(1-1/k)f(x)$ for $k>1$. Note that $f_k$ is a Lipschitz contraction on $V$, and hence $f_k$ has a unique fixed point. Moreover, for each neighborhood 
$U(f,A,\epsilon):=\{h\in N(V,\|\cdot\|)\colon \sup_{x\in A}\|h(x)-f(x)\|<\epsilon\}$ of $f$ we have that $f_k\in U(f,A,\epsilon)$ for all $k>1$ sufficiently large. 
This completes the proof of the first part of the corollary. 

Now suppose that $\Fix(f)$ is unbounded. By Theorem \ref{thm:main}(3), there exists $y \in V$, $y \neq 0$, such that $w-f(w) \neq \lambda y$ for all $w \in V$ and $\lambda > 0$. Fix some $\delta>0$ and define $g\colon V\to V$ by $g(x):=f(x)+\delta y$ for all $x\in V$. Then $\|g(x)-f(x)\|\leq \delta$ for all $x\in V$ and $\Fix(g)$ is empty, as otherwise there exists $x^*\in V$ with $x^*-f(x^*) =\delta y$, which would be a contradiction. 
\end{proof}

\section{Detecting fixed points by illumination} \label{sec:norm}
Theorem \ref{thm:main} suggests the following test for detecting fixed points of nonexpansive maps $f\colon V\to V$ on finite dimensional normed spaces. Randomly generate a finite set of points $S$ in $V$. Subsequently check if $\{f(w)-w\colon w\in S\}$ illuminates the unit ball $B_1$. If so, $f$ must have a fixed point and its fixed point set is bounded. As we shall see in this section there are many classes of norms, such as smooth norms and polyhedral norms, for which one can find computational criteria to check whether  $\{f(w)-w\colon w\in S\}$ illuminates the unit ball $B_1$. 

We start with the following basic observation. 
\begin{lemma}\label{lem:4.1} Let $K\subseteq V$ be a compact, convex set with nonempty interior. If $x,y\in\partial K$ are such that the line 
segment $\{tx+(1-t)y\colon 0\leq t\leq 1\}\subseteq \partial K$, and $x$ is illuminated by $v$, then $v$ illuminates $tx+(1-t)y$ for all $0< t\leq 1$. Moreover, if $S$ illuminates every extreme point of $K$, then $S$ illuminates $K$.
\end{lemma}
\begin{proof}
If $v$ illuminates $x$, then $x+\lambda v\in\Int K$ for all $\lambda >0$ sufficiently small. By convexity of $K$ we find that $tx+(1-t)y+t\lambda v=t(x+\lambda v)+(1-t)y\in\Int K$ for all $0<t\leq 1$. 
The second assertion now follows from the fact that each $x\in\partial K$ is a convex combination of extreme points of $K$, see \cite[Corollary 18.5.1]{Rockafellar}. 
\end{proof}

For norms with a polyhedral unit ball it is easy to check whether the extreme points of the unit ball are illuminated, which is sufficient by Lemma \ref{lem:4.1}. This leads to the following simple criterion in the case of the \textit{supremum norm} $\|x\|_\infty := \max_{1 \le i \le n} |x_i|$ on $\R^n$.   
\begin{proposition}\label{prop:supNorm} Let $f\colon\R^n\to\R^n$ be nonexpansive with respect to the supremum norm. $\Fix(f)$ is nonempty and bounded if and only if for each $J\subseteq \{1,\ldots,n\}$ there exists $w\in S$ such that 
$f(w)_j<w_j$ for all $j\in J$ and $f(w)_j>w_j$ for all $j\not\in J$.
\end{proposition}
\begin{proof}
The extreme points of the unit ball in $(\R^n,\|\cdot\|_\infty)$ are the points $z^J$ where $J \subseteq \{1,\ldots,n\}$ and $z^J_j = 1$ if $j \in J$ and $z^J_j = -1$ otherwise. 
Clearly $z^J+\lambda v\in \Int B_1$ if and only if $v_j<0$ for all $j\in J$ and $v_j>0$ for all $j\not\in J$. Thus, $\{f(w)-w\colon w\in S\}$ illuminates $B_1$ if and only if for each $J\subseteq \{1,\ldots,n\}$ there exists $w\in S$ such that $f(w)_j<w_j$ for all $j\in J$ and $f(w)_j>w_j$ for all $j\not\in J$. Therefore the result follows from Theorem \ref{thm:main}.
\end{proof}

The following necessary condition for a set to illuminate the unit ball is also sufficient for smooth norms.  The observation is closely related to known results, see \cite[Corollary 35.3]{BoltyanskiMartiniHorst}.  
\begin{lemma} \label{lem:smoothIllum}
Let $(V,\|\cdot\|)$ be a finite dimensional normed space.  If $\{v_1,\ldots,v_m\} \subset V$ illuminates the unit ball $B_1$, then $0 \in \inter \conv \{v_1, \ldots, v_m \}$.  If $\|\cdot \|$ is a smooth norm, the converse also holds. That is, $0 \in \inter \conv \{v_1, \ldots, v_m \}$ implies that $\{v_1,\ldots,v_m \}$ illuminates $B_1$.  
\end{lemma}
\begin{proof}
Suppose that $0 \notin \inter \conv \{v_1, \ldots, v_m \}$.  By the Hahn-Banach theorem, there is a linear functional $\phi \in V^*$, $\|\phi\|_* = 1$ such that $\phi(v_i) \ge 0$ for all $1 \le i \le m$.  There exists $z \in \partial B_1$ such that $\phi(z) = 1$.  Note that $\phi(z+\epsilon v_i) \ge \phi(z)$ for all $\epsilon > 0$ and $1 \le i \le m$.  Therefore $B_1$ is not illuminated by $\{v_1,\ldots,v_m\}$.  

Now suppose that $\|\cdot \|$ is a smooth norm and $0 \in \inter \conv \{v_1,\ldots, v_m \}$.  Choose $z \in \partial B_1$. By the smoothness of $\|\cdot\|$, there exists a unique $\phi \in V^*$, $\|\phi\|_*=1$, such that $\phi(z) = 1$.  For $\epsilon > 0$ sufficiently small, $-\epsilon z \in \inter \conv \{v_1, \ldots, v_m \}$.  Then 
$$-\epsilon = \phi(-\epsilon z) = \lambda_1 \phi(v_1) + \ldots + \lambda_m \phi(v_m).$$
This means that $\phi(v_i) < 0$ for some $i \in \{1, \ldots, m\}$.  Now consider the line $l = \{ z + t v_i \colon t \in \R \}$.  Note that $\|z+tv_i \| \ge \phi(z+tv_i) > 1$ for all $t < 0$.  If $\|z+tv_i \| < 1$ for some $t > 0$, then $v_i$ illuminates $z$.  Suppose that is not the case.  Then, by the Hahn-Banach theorem, there is a linear functional $\psi \in V^*$, $\|\psi\|_* = 1$, such that $\psi(z+tv_i) \ge \psi(x)$ for all $t \in \R$ and $x \in B_1$.  In particular, $\psi(z) = 1$, which implies that $\psi = \phi$ by the uniqueness of $\phi$. Then $\psi(z+tv_i) \ge 1$ for all $t$, and therefore $\psi(v_i) = 0$, a contradiction.  
\end{proof}

\begin{remark}
It is possible to determine whether $0 \in \inter \conv \{v_1, \ldots, v_m\}$ in polynomial time (in both $m$ and the dimension of $V$) using linear programming. Solve the linear program:

\begin{tabular}{ll}
\textbf{maximize} & $\epsilon \in \R$ \\
\textbf{subject to} & $\lambda_i - \epsilon \ge 0$ for all $i \in \{1,\ldots,m\}$, \\
\textbf{and} & $\sum_{i = 1}^m \lambda_i v_i = 0$, $\sum_{i = 1}^m \lambda_i = 1.$ 
\end{tabular}  \\
Then $0$ is in the relative interior of $\conv \{v_1, \ldots, v_m \}$ if and only if the maximum $\epsilon$ is positive.  Furthermore, when $0 \in \conv \{v_1,\ldots,v_m\}$, it follows that $0$ is in the affine hull of $\{v_1,\ldots, v_m\}$. This implies that the affine hull of $\{v_1, \ldots, v_m \}$ is the same as the span of $\{v_1, \ldots, v_m\}$ by \cite[Theorem 1.1]{Rockafellar}.  In particular, the relative interior of $\conv \{v_1,\ldots,v_m\}$ is the same as the interior when $\operatorname{span} \{v_1, \ldots, v_m \} = V$.  This can be checked quickly using a rank computation.
\end{remark}

Combining Lemma \ref{lem:smoothIllum} with Theorem \ref{thm:main} immediately gives the following result.  

\begin{corollary}\label{cor:smoothNorms} 
Let $f\colon V\to V$ be a nonexpansive map on a finite dimensional smooth normed space $(V,\|\cdot\|)$. There exist $w_1,\ldots,w_m\in V$ such that 
$$0\in \Int \conv\{f(w_i)-w_i\colon i=1,\ldots,m\}$$
if and only if $\Fix(f)$ is nonempty and bounded. 
\end{corollary}

For arbitrary norms the following result can be used to verify the existence of a bounded set of fixed points.  
\begin{proposition}
Let $f\colon V\to V$ be a nonexpansive map on a finite dimensional normed space $(V,\|\cdot\|)$.  If there is a finite collection of vectors $w_i \in V$ with corresponding $v_i := \frac{w_i-f(w_i)}{||w_i-f(w_i)||}$ such that the interior of the balls $B_1(v_i)$ cover $\partial B_1(0)$, then $\Fix(f)$ is nonempty and  bounded. 
\end{proposition} 
\begin{proof}
If $z$ is an extreme point of the unit ball $B_1(0)$, then $||z-v_i|| < 1$ for some $i$. Therefore $-v_i$ illuminates $z$.  Since every extreme point is illuminated by some $-v_i$, Theorem \ref{thm:main} implies that $f$ has a nonempty bounded set of fixed points.  
\end{proof}

It is worth asking whether condition (4) of Theorem \ref{thm:main} is optimal in the following sense: Suppose that the illumination number of the unit ball $B_1$ of $(V,\|\cdot\|)$ is $K$, and $v_1,\ldots,v_m\in V$ with $m<K$. Does there exist a nonexpansive map $f\colon V\to V$ and points $w_1,\ldots,w_m\in V$ with $v_i=f(w_i)-w_i$ for all $i$, such that $\Fix(f)$ is empty or unbounded?
We have the following partial results for this problem.
\begin{proposition}
If $v_1, \ldots, v_m$ are $m$ points in $\R^n$ and $m < 2^n$, then there exist a supremum-norm nonexpansive map $f\colon \R^n \rightarrow \R^n$ and  $w_i \in \R^n$ such that $v_i = f(w_i) -w_i$ for all $i$, and $\Fix(f)$ is  unbounded.
\end{proposition}
\begin{proof}
Suppose that $m < 2^n$.  For each $i$, let $w_i:=-v_i$ and  define $f(w_i) := 0$. So, $v_i = f(w_i)- w_i$ for all $i$.  Let $u := (1, \ldots, 1) \in \R^n$. Since $m < 2^n$, there must be one sign pattern in $\{-1,1\}^n$ that is not equal to the entry-wise sign pattern of any $v_i$. (Here take the sign of $0$ to be positive.) Without loss of generality, assume this is the all positive sign pattern. Define $f(c u) := cu$ for all $c \ge 0$. Since $||w_i - c u||_\infty \ge c$ and $||f(w_i) - cu||_\infty = c$ for all $i$, it follows that $f$ is nonexpansive under $||\cdot||_\infty$ on $\{w_i\colon i=1,\ldots,m\} \cup \{c u \colon c \ge 0 \}$. By a special case of the Aronszajn-Panitchpakdi theorem \cite{AronszajnPanitchpakdi}, the map $f$ extends to a supremum-norm nonexpansive map on all of $\R^n$.  
\end{proof}
For any compact, convex body $K$ in an $n$-dimensional space it is known that the  illumination number is at least $n+1$, see \cite[Theorem 35.1]{BoltyanskiMartiniHorst}. Note that this result also follows easily from Lemma \ref{lem:smoothIllum}.
So, if we are given  $v_1,\ldots,v_m$ in any $n$-dimensional normed space $V$ and $m<n+1$, we can ask if there exists a nonexpansive map $f\colon V\to V$ and points $w_1,\ldots,w_m\in V$ with $v_i=f(w_i)-w_i$ for all $i$, such that $\Fix(f)$ is empty or unbounded. The following result gives a positive answer to this question.
\begin{proposition} 
 If $v_1,\ldots,v_m$ are points in an $n$-dimensional normed space $(V,\|\cdot\|)$ and $m<n+1$, then there exist a nonexpansive map $f: V\to V$ and points $w_1,\ldots,w_m\in V$ such that $v_i = f(w_i)-w_i$ for all $i$, and $\Fix(f)$ is empty or unbounded.
\end{proposition}
\begin{proof}
First suppose that $v_1,\ldots,v_m$ span $V$. Fix $c>0$ and $w_i=-v_i$ for all $i$. Then there exists $\phi\in V^*$ such that $\phi(w_i) =c$ for all $i$.
Now let $z_0$ with $\|z_0\|=1$ be such that $\phi(z_0) =\|\phi\|_*$. So, if we let $z=z_0/\|\phi\|_*$, then 
$\phi(z)= \|\phi\|_*\|z\|=1$. 

Define $f\colon V\to V$ by $f(x)=\phi(x)z-cz$ for $x\in V$. Then 
\[
\|f(x)-f(y)\|  = |\phi(x-y)|\|z\| \leq  \|x-y\|\|\phi\|_*\|z\|=  \|x-y\|
\]
for all $x,y\in V$. Note that $f(w_i)=0$ for all $i$, so that $v_i=f(w_i) -w_i$. Also $f(\mu z)=\mu z-cz$ for all $\mu \in \R$, and hence  
$f^k(0) = -kcz$. Since $\{f^k(0)\colon k \in \N \}$ is unbounded, it follows that $\Fix(f)$ is empty. 

If $v_1,\ldots,v_m$ do not span $V$, then there exists $\psi\in V^*$ such that $\psi(v_i) =0$ for all $i$ and $\psi\neq 0$. Now let $f$ be defined as before with $\phi$ replaced by $\psi$ and $c=0$. Then $f$ is nonexpansive and $f(\mu z)=\mu z$ for all $\mu\in\mathbb{R}$. 
\end{proof}

\section{Applications to nonlinear eigenvalue problems} \label{sec:PF}
In this section we discuss applications to certain nonlinear eigenvalue problems on cones. In particular, we will consider maps $f\colon \R^n_{>0}\to  \R^n_{>0}$, where $ \R^n_{>0}$ is the interior of the standard positive cone $ \R^n_{\geq 0}:=\{x\in\R^n\colon x_i\geq 0\mbox{ for all }i\}$, that are order-preserving and homogeneous (of degree 1). Recall that $f\colon \R^n_{>0}\to  \R^n_{>0}$ is {\em order-preserving} if $f(x)\leq f(y)$ whenever $x\leq y$. Here $\leq $ is the partial ordering induced by $\R^n_{\geq 0}$, so $x\leq y$ if $y-x\in \R^n_{\geq 0}$. The map $f$ is said to be {\em homogeneous} if $f(\alpha x)=\alpha f(x)$ for all $\alpha>0$ and $x\in \R^n_{> 0}$. 

Particular motivation for studying these maps comes from game theory \cite{AkianGaubertHochart,BewleyKohlberg,RosenbergSorin} and  mathematical biology \cite{Nussbaum89,Schoen86}. In these applications it is often important to know if $f$ has an eigenvector $x\in \R^n_{> 0}$, so $f(x) =\lambda x$ for some $\lambda>0$. This is equivalent to asking whether the {\em normalized map} $g_f\colon\Sigma_0\to\Sigma_0$ given by, 
\begin{equation}\label{g_f}
g_f(x):=\frac{f(x)}{f(x)_n}\mbox{\quad for all }x\in \Sigma_0:=\{x\in \R^n_{> 0}\colon x_n=1\},
\end{equation}
has a fixed point in $\Sigma_0$. 

It is well known \cite[Lemma 2.1.6]{LNBook} that $f$ is nonexpansive under {\em Hilbert's metric}, which is given by 
\[
d_H(x,y) :=\log \left(\max_i\frac{x_i}{y_i}\right)-\log \left(\min_j\frac{x_j}{y_j}\right)\mbox{\quad for }x,y\in \R^n_{> 0}. 
\]
In fact, Hilbert's metric defines a metric between pairs of rays in $\R^n_{> 0}$, as $d_H(\alpha x,\beta y)=d_H(x,y)$ for all $\alpha,\beta>0$ and $x,y\in \R^n_{> 0}$, and $d_H(x,y) =0 $ if and only if $x=\alpha y$ for some $\alpha>0$, see \cite[Proposition 2.1.1]{LNBook}.  
Thus, $d_H$ is a metric on $\Sigma_0$ and $g_f$ is nonexpansive on $\Sigma_0$. 

If $x,y\in\R^n_{>0}$ are eigenvectors with eigenvalues say $\rho$ and $\mu$, then $\rho=\mu$, see \cite[Corollary 5.2.2]{LNBook}. It turns out that our  results can be used to analyze the {\em eigenspace},
\[\mathrm{E}(f):=\{x\in\R^n_{>0}\colon \mbox{$x$ is an eigenvector of $f$}\}.
\] 
Indeed, we have that $\mathrm{E}(f)$ is nonempty and bounded in $(\R^n_{>0},d_H)$ if and only if $\Fix(g_f)$ is nonempty and bounded in $(\Sigma_0,d_H)$. 
The reader can verify that the coordinate-wise log function is an isometry from $(\Sigma_0,d_H)$ onto the $(n-1)$-dimensional normed space $(V_0,\var{\cdot})$, where $V_0:=\{x\in\R^n\colon x_n=0\}$ and $\var{x}:=\max_i x_i-\min_j x_j$ is the {\em variation norm} on $V_0$, see \cite[\S 2.2]{LNBook}. It follows that the map $h\colon V_0\to V_0$ given by, 
\[
h(x) =(\Log\circ g_f\circ \Exp)(x),
\] 
is a nonexpansive on $(V_0,\var{\cdot})$, and hence we can apply our results to $h$. Note that the unit ball of $(V_0,\var{\cdot})$ has $2^n-2$ extreme points, which are given by, 
\begin{equation} \label{eq:varExtr}
\{v^I_+\colon \emptyset \neq I\subseteq \{1,\ldots,n-1\}\}\cup 
\{v^I_-\colon \emptyset \neq I\subseteq \{1,\ldots,n-1\}\}, 
\end{equation}
where $(v^I_+)_i =1$ if $i\in I$ and $0$ otherwise, and $(v^I_-)_i =-1$ if $i\in I$ and $0$ otherwise. See \cite[\S 2]{Nussbaum94} or \cite[Proposition 3.2]{KarlssonMetzNoskov} for details.

We begin by using our results to prove a nonlinear Perron-Frobenius type theorem. 
Recall that the classical Perron-Frobenius theorem says that if $A$ is a nonnegative $n$-by-$n$ matrix and $A$ is irreducible, then $A$ has a unique normalized eigenvector $v\in \R^n_{> 0}$ with eigenvalue the spectral radius $r(A)$ of $A$. The following result can be seen as a nonlinear Perron-Frobenius type theorem and should be compared to \cite[Theorem 6]{AkianGaubertHochart}, \cite{Cavazos},  \cite[Theorem 2]{GaubertGunawardena04}, and \cite[Theorem 6.2.3]{LNBook}. 
\begin{theorem}\label{thm:npfthm} If $f\colon \R^n_{> 0}\to \R^n_{> 0}$ is an order-preserving homogeneous map, then $\mathrm{E}(f)$ is nonempty and bounded in $(\R^n_{>0},d_H)$ if and only if for each nonempty proper subset $J$ of $\{1,\ldots,n\}$ there exists $x^J\in \R^n_{> 0}$ such that 
\begin{equation}\label{eq:5.1}
\max_{j\in J}\,\frac{f(x^J)_j}{x^J_j}< \min_{j\in J^c}\,\frac{f(x^J)_j}{x^J_j}.
\end{equation}
\end{theorem} 
\begin{proof}
Note that $\mathrm{E}(f)$ is nonempty and bounded in $(\R^n_{>0},d_H)$ if and only if $\Fix(g_f)$ is nonempty and bounded, which is equivalent to saying that the nonexpansive map $h =\Log\circ g_f\circ \Exp$ on $(V_0,\var{\cdot})$ has a nonempty and bounded fixed point set. So, if $\mathrm{E}(f)$ is nonempty and bounded, then  it follows from Theorem \ref{thm:main} that for each proper nonempty subset $J$ of $\{1,\ldots,n\}$ there exists $y^J\in V_0$ such that 
\begin{equation}\label{1}
\max_{j\in J}\, h(y^J)_j -y^J_j<\min_{j\in J^c}\, h(y^J)_j-y^J_j.
\end{equation}
Now let $x^J:=\Exp(y^J)\in \Sigma_0$. Then (\ref{1}) is equivalent to 
\[
\max_{j\in J}\,(h(\Log(x^J))_j -(\Log(x^J))_j<\min_{j\in J^c}\, (h(\Log(x^J))_j-(\Log(x^J))_j,
\]
which holds if and only if 
\begin{equation}\label{2}
\max_{j\in J}\,\log g_f(x^J)_j -\log x^J_j<\min_{j\in J^c}\, \log g_f(x^J)_j-\log x^J_j,
\end{equation}
where $g_f$ is given in (\ref{g_f}). Now note that (\ref{2}) holds if and only if 
\[
\max_{j\in J}\,\frac{g_f(x^J)_j }{x^J_j}<\min_{j\in J^c}\, \frac{g_f(x^J)_j}{x^J_j}.
\]
which is equivalent to (\ref{eq:5.1}).


Now suppose that (\ref{eq:5.1}) holds. For each nonempty proper subset $J$ of $\{1,\ldots,n\}$ let $y^J\in V_0$ be given by $y^J:=\Log(x^J/x^J_n)$. So, the inequality (\ref{1}) holds for each $y^J$. 

Note that if $v^I_+$ is an extreme point of $B_1$ given by \eqref{eq:varExtr}, then  for each $\epsilon>0$ sufficiently small we have that 
\[
\var{v^I_+ +\epsilon(h(y^I) -y^I)} = 1+\epsilon\max_{j\in I}\,( h(y^J)_j-y^J_j) -\epsilon\min_{j\in I^c}\, (h(y^J)_j-y^J_j) <1, 
\]
and hence $v^I_+$ is illuminated by $h(y^I) -y^I$. Likewise, for $v^I_-$ we can take $J:= \{1,\ldots,n\}\setminus I$, so that for all $\epsilon>0$ sufficiently small, 
\[
\var{v^I_- +\epsilon(h(y^J) -y^J)} = \epsilon\max_{j\in J}\,( h(y^J)_j-y^J_j) -(-1+\epsilon\min_{j\in J^c}\, (h(y^J)_j-y^J_j))<1,
\]
which shows that $v^I_-$ is illuminated by $h(y^J) -y^J$. It now follows from Lemma \ref{lem:4.1} and Theorem \ref{thm:main} that $\Fix(h)$ is nonempty and bounded, which implies that $\Fix(g_f)$ is nonempty and bounded.
\end{proof}
\begin{remark}
Using a case by case analysis it is not hard to show that the illumination number of the   unit ball $B_1$ in the  $(n-1)$-dimensional normed space $(V_0,\var{\cdot})$ is 3 for $n=3$, and $6$ for  $n=4$. For general $n$ the situation is not so clear, at least to the authors. The reader can, however, verify that if $S=\{v_1,\ldots,v_{2^{n-1}}\}\subseteq V_0$  is such that for each vector $s\in\{-1,1\}^{n-1}$ there exists $v\in S$ with $\sign v_i= s_i$ for all $i=1,\ldots,n-1$, where $\sign 0 =0$, then $S$ illuminates  $B_1$. So, for general $n$  the illumination number of the unit ball $B_1$ in the  normed space $(V_0,\var{\cdot})$ is at most $2^{n-1}$.  
\end{remark}

\begin{remark}
If an order-preserving homogeneous map on $\mathbb{R}^n_{>0}$ is a linear map associated to a nonnegative matrix $A$, then the eigenspace $\mathrm{E}(A)$ is nonempty and bounded in Hilbert's metric if and only if $A$ has a unique  (up to scaling) eigenvector in $\mathbb{R}^n_{>0}$. 

For a nonnegative $n$-by-$n$ matrix $A = (a_{ij})$, let $G(A)$ denote the adjacency digraph of $A$, that is, the graph on vertices $\{1,\ldots,n\}$ with an edge from $i$ to $j$ if and only if $a_{ij} > 0$.  For $i, j \in \{1,\ldots,n\}$, we say that $i$ has \textit{access} to $j$ if there is a path from $i$ to $j$ in $G(A)$. We say that $i$ and $j$ \textit{communicate} if they both have access to each other.  Communication is an equivalence relation, and the equivalence classes of $\{1,\ldots,n\}$ under communication are called the \textit{classes} of $A$.  A class $\alpha$ is \textit{final} if no vertex $i \in \alpha$ has access to any vertex outside $\alpha$. It is \textit{basic} if the square submatrix of $A$ corresponding to $\alpha$ has spectral radius equal to the spectral radius of $A$. A nonnegative matrix $A$ has a positive eigenvector if and only if the final classes of $A$ are exactly its basic classes. Furthermore, the positive eigenvector is unique (up to scaling) if and only if $A$ has only one basic, final class. See \cite[Ch. 2, Theorem 3.10 and its proof]{BermanPlemmons} for details. Note that the classes and their corresponding spectral radii can be determined with prescribed accuracy in polynomial time as the dimension $n$ grows.  This is much faster than verifying the conditions of Theorem \ref{thm:npfthm} for linear maps with large $n$.  
\end{remark}


We should note that a nonlinear order-preserving homogeneous map $f\colon \mathbb{R}^n_{>0}\to\mathbb{R}^n_{>0}$ can have an eigenspace space $E(f)$ which is bounded in Hilbert's metric and consists of more than a single ray.  Simple examples of such maps can be constructed as follows. 
\begin{example} \label{ex:tri}
For $0 \leq c \leq \tfrac{1}{3}$, let 
\begin{align*}
\mu_{1,c}(x) := \max\{x_2,x_3,c(x_1+x_2+x_3)\}, \\
\mu_{2,c}(x) := \max\{x_1,x_3,c(x_1+x_2+x_3)\}, \\
\mu_{3,c}(x) := \max\{x_1,x_2,c(x_1+x_2+x_3)\}.
\end{align*}
Let $f_c\colon \R^3_{>0} \to \R^3_{>0}$ be defined
$$
f_c(x) := \begin{cases}
(x_1,\mu_{1,c}(x),\mu_{1,c}(x)) & \text{if }x_1 = \max\{x_1,x_2,x_3\}, \\
(\mu_{2,c}(x),x_2,\mu_{2,c}(x)) & \text{if }x_2 = \max\{x_1,x_2,x_3\}, \\
(\mu_{3,c}(x),\mu_{3,c}(x),x_3) & \text{if }x_3 = \max\{x_1,x_2,x_3\}. \\
\end{cases} 
$$
It is not hard to verify that $f_c$ is well-defined, order-preserving, and homogeneous. In the case $c=0$, $f_c$ was described in \cite[p.\,131]{Nussbaum89}. Let $\Sigma := \{ x \in\R^3_{>0}\colon  x_1+x_2+x_3 = 1 \}$ and let $g_c(x) := f_c(x)/(x_1+x_2+x_3)$ for all $x \in \Sigma$. For convenience, let $e_1, e_2, e_3$ denote the elementary basis vectors in $\R^n$, so that $\Sigma$ is the relative interior of $\conv \{e_1, e_2, e_3 \}$.  For each $c$, $\Fix(g_c)$ is the union of three line segments, $[\tfrac{1}{3}(e_1+e_2+e_3), (1-3c)e_i + c (e_1+e_2+e_3)] \cap \Sigma$, $i \in \{1,2,3\}$ as shown in Figure \ref{fig:fluxCap}.  If $c= \frac{1}{3}$, then $\Fix(g_c)$ is a single point and therefore $f_c$ has a unique eigenvector in $\R^3_{>0}$ (up to scaling). When $c < \tfrac{1}{3}$, however, the fixed points of $g_c$ are not unique, nor are the fixed point sets convex. If  $c > 0$, $\Fix(g_c)$ is nonempty and bounded in Hilbert's metric on $\Sigma$. If $c = 0$, the fixed point set is unbounded and Theorem \ref{thm:npfthm} does not apply.  
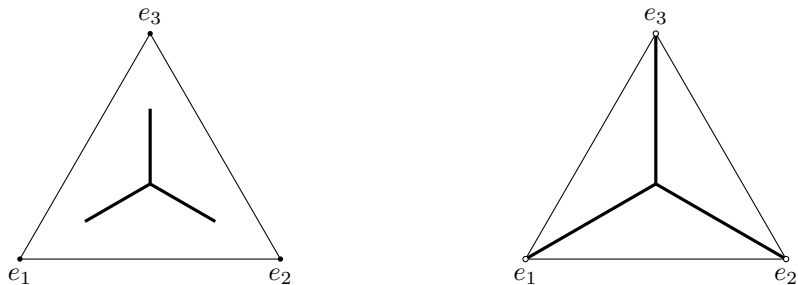
\begin{figure}[h]
\[
\begin{tabular}{ccc}
\begin{tikzpicture}
\draw (90:2) -- (210:2) -- (330:2) -- cycle;
\draw[very thick] (0,0) -- (90:1);
\draw[very thick] (0,0) -- (210:1);
\draw[very thick] (0,0) -- (330:1);
\fill (90:2) node[above] {$e_3$} circle (1pt); 
\fill (210:2) node[below] {$e_1$} circle (1pt);
\fill (330:2) node[below] {$e_2$} circle (1pt);
\end{tikzpicture} 
& \hspace*{2cm} &
\begin{tikzpicture}
\draw (90:2) -- (210:2) -- (330:2) -- cycle;
\draw[very thick] (0,0) -- (90:2);
\draw[very thick] (0,0) -- (210:2);
\draw[very thick] (0,0) -- (330:2);
\filldraw[fill=white] (90:2) node[above] {$e_3$} circle (1pt); 
\filldraw[fill=white] (210:2) node[below] {$e_1$} circle (1pt);
\filldraw[fill=white] (330:2) node[below] {$e_2$} circle (1pt);
\end{tikzpicture}
\end{tabular}\]
\caption[]{$\Fix(g_c)$ from Example \ref{ex:tri} when $c = \tfrac{1}{6}$ (left) and $c = 0$ (right).} \label{fig:fluxCap}
\end{figure}
\end{example}


Theorem \ref{thm:npfthm} suggests the following algorithm for detecting the presence of eigenvectors in $\mathbb{R}^n_{>0}$.  Let $f\colon \mathbb{R}^n_{>0} \to \mathbb{R}^n_{>0}$ be an order-preserving homogeneous  map .  
\begin{description}
\item[Step 1] Randomly select $x \in \Sigma_0$ and compute $f(x)_j/x_j$ for all $j \in \{1,\ldots, n\}$.  
\item[Step 2] Record all nonempty proper subsets $J \subset \{1,\ldots,n\}$ such that inequality  (\ref{eq:5.1}) holds.  
\item[Step 3] Repeat steps 1 and 2 until every nonempty proper subset $J$ has been recorded.  
\end{description}

If the algorithm above halts, then $f$ has an eigenvector in $\mathbb{R}^n_{>0}$. Of course, if the eigenspace of $f$ is  empty or unbounded under Hilbert's metric, then the algorithm will never halt. Also note that as each $x\in \mathbb{R}^n_{>0}$ can satisfy inequality (\ref{eq:5.1}) for at most $n-1$ different proper non-empty subsets $J$ of $\{1,\ldots,n\}$,  the algorithm will need to try at least $\frac{2^n-2}{n-1}$ different elements of $\Sigma_0$, which makes it impractical when $n$ is large. However, in low dimensional spaces the algorithm could be useful.

In the remainder of this section, we describe a class of order-preserving homogeneous maps $f\colon \mathbb{R}^n_{>0}\to\mathbb{R}^n_{>0}$ for which no general methods exist to determine the existence of an eigenvector in $\mathbb{R}^n_{>0}$. For these maps, the algorithm above may be particularly useful.  For $r\in\mathbb{R}$ with $r\neq 0$, and $\sigma\in \mathbb{R}^n_{\geq 0}$ with $\sum_i\sigma_i=1$ define the {\em $(r,\sigma)$-mean} of $x\in\mathbb{R}^n_{>0}$ by 
\[
M_{r\sigma}(x):=\left(\sum_{i=1}^n \sigma_ix_i^{1/r}\right)^{1/r},
\]
and $M_{0\sigma}(x):= \prod_{i\in\mathrm{supp}(\sigma)} x_i$, where $\mathrm{supp}(\sigma)=\{i\colon \sigma_i>0\}$. Furthermore define 
$M_{\infty,\sigma}(x) :=\max\{x_i\colon i\in\mathrm{supp}(\sigma)\}$ and $M_{-\infty,\sigma}(x) :=\min\{x_i\colon i\in\mathrm{supp}(\sigma)\}$. We say that an order-preserving homogeneous map $f\colon \mathbb{R}^n_{>0}\to\mathbb{R}^n_{>0}$ belongs to $M$ if each coordinate function is of the form 
\[
f_i(x) =\sum_{(r,\sigma)\in\Gamma_i} c_{ir\sigma}M_{r\sigma}(x),
\]
where $\Gamma_i$ is a nonempty set of pairs $(r,\sigma)$, with $r\in [-\infty,\infty]$ and $\sigma\in \mathbb{R}^n_{\geq 0}$ such that $\sum_i\sigma_i=1$, and each $c_{ir\sigma}>0$. A map $f\in M$ is said to belong to $M_+$ if each $r\in [0,\infty)$, and it belongs to $M_-$ if each $r\in (-\infty,0)$. By closing the sets $M$, $M_+$ and $M_-$ under multiplication with positive scalars, addition, and composition, we obtain classes of maps $\mathcal{M}$, $\mathcal{M}_+$, and $\mathcal{M}_-$, respectively. For maps in  $\mathcal{M}_+$ there exist a variety of general results to determine the existence of an eigenvector in $\mathbb{R}^n_{>0}$. However, no general methods for detecting eigenvectors in $\mathbb{R}^n_{>0}$ are known for maps in $\mathcal{M}_-$. An extensive discussion of this problem can be found in \cite[Section 6.6]{LNBook} and \cite{Nussbaum89}. 


In \cite{Schoen86}, the following map $f\colon \R^4_{>0} \to \R^4_{>0}$ was studied as part of a population biology model,  
\begin{equation} \label{eq:schoen}
f(x) := \left( \begin{array}{l} 
a_1 x_1 + b_1 \theta(x_1,x_2) + c_1 \theta(x_1,x_4) + d_1 \theta(x_2,x_3) \\
a_2 x_2 + b_2 \theta(x_1,x_2) + c_2 \theta(x_1,x_4) + d_2 \theta(x_2,x_3) \\
a_3 x_3 + b_3 \theta(x_3,x_4) + c_3 \theta(x_1,x_4) + d_3 \theta(x_2,x_3) \\
a_4 x_4 + b_4 \theta(x_3,x_4) + c_4 \theta(x_1,x_4) + d_4 \theta(x_2,x_3) \end{array}
\right)
\end{equation}
where the coefficients $b_i, c_i, d_i \ge 0$ for $ 1 \le i \le 4$, with at least one positive for each $i$, and $\theta(s,t) := (s^{-1}+t^{-1})^{-1}$. In the original model, the coefficients $a_i$ were negative as they represent a `force of mortality'. By adding a multiple of the identity to $f$, we may assume that each $a_i$, $1 \le i \le 4$, is positive without changing the eigenvectors of $f$.  With this additional assumption, $f \in M_-$.  In \cite[\S 3]{Nussbaum89}, detailed conditions on the coefficients of $f$ are given that determine whether or not $f$ has an eigenvector with positive entries.  

Unlike the conditions in \cite{Nussbaum89}, the algorithm discussed in this section does not classify all coefficients for which the map above has a entry-wise positive eigenvector.  For any particular choice of coefficients, however, the algorithm gives an elementary method for determining whether the map with those coefficients has a positive eigenvector. The advantage of the new algorithm lies in its ability to work with any order-preserving homogeneous map $f\colon \R^n_{>0} \to \R^n_{>0}$, even ones for which other techniques fail. 

\begin{example}
Let $f$ and $g$ be defined by \eqref{eq:schoen}, where the coefficients of $f$ and $g$ are given by  
$$\begin{pmatrix}
a_1 & b_1 & c_1 & d_1  \\
a_2 & b_2 & c_2 & d_2 \\
a_3 & b_3 & c_3 & d_3 \\
a_4 & b_4 & c_4 & d_4 
\end{pmatrix} := \begin{pmatrix}
1 & 2 & 3 & 4  \\
2 & 1 & 1 & 1 \\
3 & 1 & 3 & 5 \\
4 & 3 & 1 & 2 
\end{pmatrix}, \text{ and } \begin{pmatrix}
2 & 5 & 7 & 2  \\
3 & 3 & 1 & 1 \\
4 & 4 & 13 & 1 \\
1 & 2 & 7 & 8
\end{pmatrix},$$
respectively.  The composition $f \circ g \in \mathcal{M_-}$.  Known results cannot confirm whether or not $f \circ g$ has an eigenvector in $\R^4_{>0}$.  However, our algorithm quickly verifies that $f \circ g$ has eigenvectors in $\R^4_{>0}$.  Note that $f \circ g$ is Fr\'echet differentiable and the derivative $D(f\circ g)(x)$ is a matrix with all positive entries for every $x \in \R^4_{>0}$.  This implies that if $f \circ g$ has an eigenvector in $\R^4_{>0}$, then it must be unique up to scaling.  See, for example, \cite[Corollary 6.4.8]{LNBook}. It also follows from \cite[Theorem 3.7]{Nussbaum89} that for any strictly positive vector $x$ in $\R^4$, the normalized iterates of $f\circ g$ applied to $x$ converge to the unique normalized eigenvector of $f \circ g$ once one knows that $f \circ g$  has a strictly positive eigenvector. This  observation can be exploited to make our algorithm somewhat more efficient for many functions in $\mathcal{M}$, though we have not attempted to do so here. 

To generate test vectors for the algorithm, we randomly selected vectors $w = \operatorname{Exp}(y)$, where $y \in V_0:=\{x\in\mathbb{R}^n\colon x_n=0\}$, are vectors that are uniformly distributed in the set $\{y \in V_0 \colon -R \leq y_j \leq R \mbox{ for all } 1 \leq j \leq n-1 \}$.  For this example, the value $R= 100$ worked well, but of course, $R$ must be chosen large enough to accommodate the set of eigenvectors of the map.  We ran 500 independent trials and recorded the number of test vectors $w$ needed to confirm the existence of a bounded set of eigenvectors. The largest number of test vectors needed was 303, and the smallest was 10.  The average was 54.4, and the median was 39. By iterating $f \circ g$ on the vector $(1,1,1,1) \in \R^4$, we find that the unique eigenvector is approximately $(0.24138896, 0.10237913, 0.56235034, 1)$ when normalized so that the last entry is 1.  

\end{example}


\section{Localizing the fixed point set}

Once the presence of a nonempty and bounded set of fixed points has been confirmed, a natural follow up problem is to give bounds on the location of the fixed points. Here we show how this can be accomplished.
 
If $C$ is a bounded set in a finite dimensional normed linear space $V$, define $R_0 \ge 0$ by 
$$R_0 := \inf \{R>0 \colon \bigcap_{x \in C} B_R(x) \ne \varnothing \}$$
and call $R_0$ the \emph{circumradius of} $C$.  Note that as for each $R>R_0$, $\bigcap_{x\in C} B_R(x)$ is compact, convex and nonempty, therefore $\bigcap_{x \in C} B_{R_0}(x)$ is compact, convex and nonempty. A point $p \in \bigcap_{x \in C} B_{R_0}(x)$ is called a \emph{circumcenter} of $C$.  

Suppose that $(V,\|\cdot\|)$ is a finite dimensional normed space with unit ball $B_1$ and suppose that $\{v_1,\ldots,v_m \}$ illuminate $\partial B_1$. 
For each $j =1,\ldots,m$, the set $U_j:=\{x\in\partial B_1\colon v_j \text{  illuminates }x \}$ is relatively open in $\partial B_1$. Let $U_j^c := \partial B_1 \backslash U_j$ and note that $\max_j d(z,U^c_j)>0$ for all $z\in\partial B_1$. As $\partial B_1$ is compact, we find that 
\begin{equation} \label{eq:delta}
\delta:= \min_{z\in\partial B_1}\max_{j=1,\ldots,m} d(z,U^c_j)>0.
\end{equation}  
We also define the following two constants.  
\begin{equation} \label{eq:A}
\alpha := \sup_{z \in \partial B_1} d(z,\extr{B_1}).
\end{equation}
\begin{equation} \label{eq:B}
\beta := \inf \{ \|z-u\| \colon z \in \partial B_1, u \in \extr{B_1}  \text{ and } \|\tfrac{1}{2}z+\tfrac{1}{2}u\| < 1 \}.
\end{equation}
Here $\extr B_1$ denotes the extreme points of $B_1$.  
\begin{theorem} \label{thm:localize}
Suppose $f\colon V \rightarrow V$ is a nonexpansive map on a finite dimensional normed space $(V,\|\cdot\|)$ with unit ball $B_1$ and there exist $w_1, \ldots, w_m \in V$ such that $\{f(w_i) - w_i\colon i = 1, \ldots, m \}$ illuminates $B_1$.  If $R_0$ is the circumradius of $\{w_1, \ldots, w_m \}$, $p$ is a circumcenter and $\delta$ is defined as in \eqref{eq:delta}, then $\Fix(f) \subset B_R(p)$ where $R := (\tfrac{2+\delta}{\delta})R_0$.  Furthermore, if $\alpha,\beta$ are defined as in \eqref{eq:A} and \eqref{eq:B} above and $\beta > \alpha$, then $\delta \ge \beta-\alpha$ and $\Fix(f) \subset B_{R'}(p)$ where $R' := (\tfrac{2+\beta-\alpha}{\beta-\alpha})R_0$. 
\end{theorem}

\begin{proof}
Let $R_0$ be the circumradius of $\{w_i\colon i = 1, \ldots, m\}$ and let $p$ be a circumcenter for $\{w_i\colon i = 1, \ldots, m\}$, so $\|w_i - p \| \le R_0$ for $ 1 \le i \le m$.  Define $\hat{f}(x) := f(x+p) - p$ and $\hat{w}_i := w_i - p$ for $i = 1,\ldots,m$, so $\hat{f}$ is nonexpansive, $\hat{f}(\hat{w}_i) - \hat{w}_i = f(w_i)-w_i$ and $\|\hat{w}_i\| \le R_0$ for $1 \le i \le m$  and $\Fix(\hat{f}) = \{x-p \colon x \in \Fix(f)\}$.  Thus, by replacing $f$ with $\hat{f}$ and $w_i$ by $\hat{w}_i$, $1\le i \le m$, we may as well assume that $p=0$ and $\|w_i \| \le R_0$ for $1 \le i \le m$.  

Suppose now that $z \in V$ and $\|z \| > R:= (\frac{2+\delta}{\delta})R_0$. By \eqref{eq:delta}, there is an $i \in \{1,\ldots, m\}$ such that $d(z/\|z\|, U_i^c) \ge \delta$.  

We claim that, 
$$\left\| \frac{z-w_i}{\|z-w_i\|} - \frac{z}{\|z\|} \right\| < \delta.$$
To see this, observe that 
\begin{align*}
\left\| \frac{z-w_i}{\|z-w_i\|} - \frac{z}{\|z\|} \right\| &\le \frac{\|w_i\|}{\|z-w_i\|} + \|z\| \left| \frac{1}{\|z-w_i\|}- \frac{1}{\|z\|} \right| \\ 
&\le \frac{R_0}{\|z\|-R_0} + \frac{| \|z\| - \|z-w_i\| |}{\|z-w_i\|} \\
&\le \frac{R_0}{\|z\|-R_0} + \frac{R_0}{\|z\|-R_0}  \\
&< \frac{2R_0}{(\frac{2+\delta}{\delta})R_0-R_0} = \delta.
\end{align*}
It follows from \eqref{eq:delta} that $\frac{z-w_i}{\|z-w_i\|}$ is illuminated by $f(w_i)-w_i$. Therefore 
\[
\frac{z-w_i}{\|z-w_i\|} +\lambda (f(w_i)-w_i) \in\Int B_1
\]
for some small $\lambda > 0$. Note that the function $\mu \colon \R \to \R$ defined by
$$\mu\colon t \mapsto \left\| \frac{z-w_i}{\|z-w_i\|} +t(f(w_i)-w_i)\right\|$$
is convex. Since $\mu(0) = 1$ and $\mu(\lambda) < 1$, it follows that $\mu(t) > 1$ for all $t < 0$.  In particular, when $t = -1/\|z-w_i\|$ we see that 
\[
\left\| \frac{z-w_i}{\|z-w_i\|} -\frac{f(w_i) -w_i}{\|z-w_i\|}\right\| > 1 
\]
so that 
\[
\|z-f(w_i)\|>\|z-w_i\|.
\]
If $f(z) = z$, the inequality above contradicts the nonexpansiveness of $f$.  Thus $\Fix(f) \subset B_R(0)$. Note that the inclusion is strict as $B_R(0)$ contains $\{w_1,\ldots,w_m\}$, and none of the $w_i$ are fixed points of $f$. 

Now suppose that $\alpha, \beta$ are defined as in \eqref{eq:A} and \eqref{eq:B}, and $\beta > \alpha$.  Choose any $z \in \partial B_1$.  We will show that $\delta \ge \beta-\alpha$ by showing that
\begin{equation} \label{eq:maxDelta}
\max_{1 \le i \le m} d(z,U_i^c) \ge \beta - \alpha.
\end{equation}
First note that there exists an extreme point $u \in B_1$ such that $\|z-u\| \le \alpha'$ for any $\alpha'$ satisfying $\alpha < \alpha' < \beta$. There is also an $i \in \{1,\ldots,m\}$ such that $f(w_i)-w_i$ illuminates $u$. Now consider any $y \in U_i^c$.   Note that $v_i := f(w_i)-w_i$ does not illuminate $y$, therefore for all $\epsilon >0$, $\|y+\epsilon v_i \| \ge 1$.  Then for all sufficiently small $\epsilon >0$,
$$\frac{1}{2} \frac{y+\epsilon v_i}{\|y+\epsilon v_i\|} + \frac{1}{2}u = \frac{1}{2} \frac{y}{\|y+\epsilon v_i \|}+\frac{1}{2} \left(u+\frac{\epsilon v_i}{\|y+\epsilon v_i\|} \right) \in \inter B_1.$$
By \eqref{eq:B}, it follows that 
$$\left\| \frac{y+\epsilon v_i}{\|y+\epsilon v_i\|} - u \right\| \ge \beta.$$
By taking the limit as $\epsilon \rightarrow 0$, we see that $\|y-u\| \ge \beta$ as well.  By the triangle inequality,
$$\|z-y\| \ge \|y-u\| - \|z-u\| \ge \beta - \alpha'.$$
By letting $\alpha'$ approach $\alpha$, we complete the proof of \eqref{eq:maxDelta}. Moreover, 
$$R' = \left(\frac{2+\beta-\alpha}{\beta-\alpha}\right)R_0 \ge \left(\frac{2+\delta}{\delta}\right)R_0=R,$$ 
so $\Fix(f) \subset B_{R'}(0)$.   
\end{proof}

\begin{remark}
For general norms, it is not possible to place uniform lower bounds on $\delta$ from \eqref{eq:delta} without additional assumptions on the illuminating set. However, for some polyhedral norms the constants $\alpha, \beta$ from \eqref{eq:A} and \eqref{eq:B} satisfy $\beta-\alpha > 0$, and it is possible to give uniform bounds on $\Fix(f)$ based only on the circumcenter and circumradius of a set $\{w_1,\ldots, w_m\} \subset V$ such that $\{f(w_i)-w_i \colon 1 \le i \le m\}$ illuminates the unit ball.  
 
For the supremum norm $\|\cdot\|_\infty$ on $\R^n$, it is not hard to verify that the constants in \eqref{eq:A} and \eqref{eq:B} are $\alpha = 1$ and $\beta = 2$.  So, if $f:\R^n \rightarrow \R^n$ is nonexpansive with respect to $\|\cdot\|_\infty$, $\{f(w_i) - w_i \colon 1 \le i \le m \}$ is a set that illuminates the unit ball in $(\R^n, \|\cdot\|_\infty)$, $R_0$ is the circumradius of $\{w_i\colon 1\le i \le m \}$, and $p$ is a circumcenter of $\{w_i\colon 1\le i \le m \}$, then $\Fix(f) \subset B_{3R_0}(p)$.  

For the $l_1$ norm on $\R^n$, $\alpha = 2-2/n$ and $\beta = 2$. We leave the details as an exercise for the reader.  
\end{remark}

For inner-product spaces the following  can be shown. 
\begin{proposition} 
If $f\colon V\to V$ is a nonexpansive map on an inner-product space $V$, then 
\[
\Fix(f)\subseteq \bigcap_{w\in V} H_w,
\]
where $H_w:=\{v\in V\colon \langle v, w-f(w)\rangle \leq \langle w, w-f(w)\rangle\}$. Moreover, if $V$ is finite dimensional  and $\{f(w_i)-w_i\colon i=1,\ldots,m\}$ illuminates the unit ball of $V$, then $\Fix(f)\subseteq \bigcap_{i=1}^m H_{w_i}$ and $\bigcap_{i=1}^m H_{w_i}$ is a compact set. 
\end{proposition}
\begin{proof}
Suppose that there exist $x\in\Fix(f)$ and $w\in V$ such that $x\not\in H_w$. Since $f$ is nonexpansive, $\|x-f(w)\|=\|f(x)-f(w)\|\leq \|x-w\|$, so that 
\begin{align*}
0 &\geq \|x-f(w)\|^2-\|x-w\|^2\\
 &= 2\langle x, w-f(w)\rangle +\|f(w)\|^2-\|w\|^2\\
  &> 2\langle w, w-f(w)\rangle +\|f(w)\|^2-\|w\|^2\\
 &=  \|w-f(w)\|^2, 
\end{align*}
which is impossible. 

To prove the second part note that it follows from Lemma \ref{lem:smoothIllum}  that $\conv \{f(w_i)-w_i \colon i=1,\ldots,m\}$ is a compact polytope with 0 in its interior. 
This implies that $P:=\conv \{w_i- f(w_i)\colon i=1,\ldots,m\}$ is a compact polytope with 0 in its interior. The polar of P is given by $P^\circ := \{v \in V \colon \langle v,w_i-f(w_i)\rangle \leq 1\mbox{ for all } i=1,\ldots,m\}$,  which  is also  a compact polytope with 0 in interior. 
Now let \[
Q : = \{v \in V \colon \langle v,w_i-f(w_i)\rangle \leq C\text{ for all } i=1,\ldots,m\},
\] 
where $C:= \max\{1, \max_{i=1,\ldots,m}\langle w_i, w_i -f(w_i)\rangle\}$, and note that $Q$ is a compact polytope with 0 in interior. 
Clearly, 
\[
\bigcap_{i=1}^m \{v \in V \colon \langle v,w_i-f(w_i)\rangle \leq \langle w_i,w_i-f(w_i)\rangle \} \subseteq Q,
\]
which completes the proof.
\end{proof}

By applying Theorem \ref{thm:localize} to variation norm nonexpansive maps we derive the following result.  

\begin{theorem} \label{thm:eigloc}
Suppose $f\colon \R^n_{> 0}\to \R^n_{> 0}$ is an order-preserving homogeneous map and for each nonempty proper subset $J$ of $\{1, \ldots, n \}$ there is an $x^J \in \R^n_{>0}$ satisfying \eqref{eq:5.1}. If $R_0$ is the Hilbert metric circumradius of the set $S:=\{x^J\colon J \subset \{1,\ldots,n\}, J \ne \varnothing\}$, and $p$ is a circumcenter of $S$, then $\mathrm{E}(f) \subset B_{(2n-1)R_0}(p)$. 
\end{theorem}

To prove this theorem we need the following lemma. 
\begin{lemma}\label{lem:5.4}
Let $B_1$ be the unit ball in the $(n-1)$-dimensional normed space $(V_0,\var{\cdot})$. 
\begin{enumerate}
\item[(a)] For each $w\in\partial B_1$ there exists an extreme point $v$ of $B_1$ with 
\[
\var{v-w}\leq 1-\frac{1}{n-1}.
\] 
\item[(b)] If $v \in \extr B_1$ and $w\in\partial B_1$ with $\var{v-w}<1$, then $\frac{1}{2}v+\frac{1}{2} w \in \partial B_1$.   
\end{enumerate}
\end{lemma}
\begin{proof}
Let $\R e=\mathrm{span}\{e\}$, where $e=(1,\ldots,1)$. On the quotient space $\R^n/\R e$ we also have the variation norm $\var{[x]}=\max_{1\leq i\leq n} x_i -\min_{1\leq i\leq n}x_i$ for $[x]\in \R^n/\R e$.  It is easy to verify that $L\colon \R^n/\R e\to V_0$ given by $L[x] = x-x_ne$ is a well-defined linear isometry from $(\R^n/\R e,\var{\cdot})$ onto $(V_0,\var{\cdot})$. Thus, it suffices to show the assertions for the unit ball $B_1$ in $(\R^n/\R e,\var{\cdot})$.

To prove part (a) let $[w]\in\R^n/\R e$ with $\var{[w]}=1$. We assume without loss of generality that $1=w_1\geq w_2\geq\ldots\geq w_n=0$, as we can relabel the coordinates. Thus, there exists $i<n$ such that 
\[
w_i-w_{i+1}\geq \frac{1}{n-1}.
\]
The extreme points of $B_1$ in $(\R^n/\R e,\var{\cdot})$ are given by $\{[v]\colon v\in\{0,1\}^n\}\setminus\{[e]\}$. Now let $[v]$ be the extreme point with $v_k=1$ for all $k\leq i$ and $v_i=0$ otherwise. Then 
\[
\var{[v]-[w]} =\var{v-w} = 1-w_i -(-w_{i+1}) \leq 1-\frac{1}{n-1},
\]
which proves part (a). 

To prove the second assertion  let $[v]$ be an extreme point of $B_1$. By relabeling the coordinates, we may assume that $v_k=1$ for all $k\leq i$ and 
$v_k=0$ otherwise. Now suppose that $[w]\in\partial B_1$ and $\var{[v]-[w]}<1$. We can assume that $w_k\geq 0$ for all $k$. As $\var{[v]-[w]}<1$, it follows that $0<w_1,\ldots,w_i\leq 1$ and $0\leq w_{i+1},\ldots,w_n<1$. By relabeling the coordinates of $w$ we may furthermore assume that 
$ 0<w_i\leq \ldots \leq w_1\leq 1$ and $0\leq w_{i+1}\leq \ldots\leq w_n<1$. Recall that $\var{[w]}=1$, so that $w_1 =1$ and $w_{i+1}=0$. Then $\tfrac{1}{2}(w_1+v_1) = 1$ and $\tfrac{1}{2}(w_{i+1}+v_{i+1}) = 0$, so $\var{\tfrac{1}{2}[w]+\tfrac{1}{2}[v]} \ge 1$.
\end{proof}

\begin{proof}[Proof of Theorem \ref{thm:eigloc}]
Let $\Sigma_0,g_f,V_0,h$ be as described in the proof of Theorem \ref{thm:npfthm}. As observed in the proof of Theorem \ref{thm:npfthm}, the set 
$$\{h(y^J) - y^J\colon y^J = \Log(x^J/x^J_n) \text{ where } x^J \in S \}$$ 
illuminates the unit ball in $(V_0,\var{\cdot})$. Since the map $\Log$ is an isometry from $(\Sigma_0,d_H)$ onto $(V_0,\var{\cdot})$, it suffices to prove that $\Fix(h) \subset B_{(2n-1)R_0}(\Log(p))$ in $(V_0,\var{\cdot})$. 

By Lemma \ref{lem:5.4}, the constants in \eqref{eq:A} and \eqref{eq:B} for $(V_0,\var{\cdot})$ are $\alpha = 1-\tfrac{1}{n-1}$ and $\beta = 1$. Theorem \ref{thm:localize} implies that $\Fix(h) \subset B_{(2n-1)R_0}(\Log(p))$ in $(V_0,\var{\cdot})$.  Therefore $E(f) \subset B_{(2n-1)R_0}(p)$ in $(\R^n_{>0},d_H)$.  
\end{proof}

\bibliography{DW}
\bibliographystyle{plain}

\end{document}